\documentclass[12pt]{amsart}

\usepackage{amssymb}

\newtheorem{theorem}{Theorem}
\newtheorem{lemma}[theorem]{Lemma}
\newtheorem{proposition}[theorem]{Proposition}

\theoremstyle{definition}
\newtheorem{definition}[theorem]{Definition}

\numberwithin{equation}{section}

\newcommand{\dbar}  {\bar \partial}

\newcommand{\tensor}{\otimes}

\newcommand{\CC}{\mathbb{C}}
\newcommand{\RR}{\mathbb{R}}

  \newcommand{\A}{\mathcal{A}}
  \newcommand{\D}{\mathcal{D}}

\renewcommand{\O}{\mathcal{O}}

\DeclareMathOperator{\degree}{degree}
\DeclareMathOperator{\End}{End}

\DeclareMathOperator{\id}{id}
\DeclareMathOperator{\pardeg}{par-deg}
\DeclareMathOperator{\parmu}{par-\mu}
\DeclareMathOperator{\rank}{rank}
\DeclareMathOperator{\Res}{Res}

\renewcommand{\leq}{\leqslant}
\renewcommand{\geq}{\geqslant}

\newcommand*{\longhookrightarrow}{\ensuremath{\lhook\joinrel\relbar\joinrel\rightarrow}}

\begin{document}

\baselineskip=15pt

\title[Orthogonal and symplectic parabolic Higgs bundles]{Hermitian-Einstein
connections on polystable orthogonal and symplectic parabolic Higgs bundles}

\author[I. Biswas]{Indranil Biswas}

\address{School of Mathematics, Tata Institute of Fundamental
Research, Homi Bhabha Road, Bombay 400005, India}

\email{indranil@math.tifr.res.in}

\author[M. Stemmler]{Matthias Stemmler}

\email{stemmler@math.tifr.res.in}



\date{}

\begin{abstract}
Let $X$ be a smooth complex projective curve and 
$S \subset X$ a finite subset. We show that an orthogonal or 
symplectic parabolic Higgs bundle on $X$ with parabolic structure
over $S$ admits a Hermitian-Einstein connection if and only if it is 
polystable.
\end{abstract}

\maketitle

\section{Introduction}

Let $X$ be an irreducible smooth complex projective curve, and let 
$S \subset X$ be a fixed finite subset. The notion of parabolic 
vector bundles on $X$ with $S$ as the parabolic divisor was 
introduced by Seshadri \cite{Se77}. Parabolic bundles equipped 
with Higgs fields were introduced by Simpson \cite{Si90} under the 
name of filtered regular Higgs bundles; see also \cite{Yo93}.

An orthogonal or symplectic parabolic bundle is a parabolic vector 
bundle equipped with a symmetric or alternating form, 
respectively, with values in a parabolic line bundle; this form
is required to be non-degenerate in a suitable sense \cite{BMW11}.
In the case of rational 
parabolic weights, this coincides with the notion of parabolic 
principal $G$-bundles introduced in \cite{BBN01} and \cite{BBN03}, 
where $G$ is the orthogonal or symplectic group, respectively.

In \cite{BMW11}, orthogonal and symplectic parabolic bundles were 
investigated. In particular, a Hitchin-Kobayashi correspondence,
which says that an 
orthogonal or symplectic parabolic bundle admits a 
Hermitian-Einstein connection if and only if it is polystable,
was established.
Our aim here is to define Higgs fields on orthogonal and symplectic parabolic bundles and to generalize the Hitchin-Kobayashi correspondence to this context. We obtain the following result (see Theorem \ref{result} and Proposition \ref{converse}):

\begin{theorem}
Let $(E_\ast\, , \varphi\, , \theta)$ be an orthogonal or symplectic
parabolic Higgs bundle. If $(E_\ast\, , \varphi\, , \theta)$ is
polystable, then it admits a Hermitian-Einstein connection.

Conversely, if $(E_\ast\, , \varphi\, , \theta)$ admits a
Hermitian-Einstein connection lying in the space $\A$ (see
\eqref{poritz}), then it is polystable.
\end{theorem}

\section{Orthogonal and symplectic parabolic Higgs bundles}

\subsection{Parabolic vector bundles}

Let $X$ be an irreducible smooth complex projective curve. Fix a finite subset
\[
  S \,:=\, \{ x_1\, , \cdots\, , x_n \} \,\subset\, X
\]
with distinct points $x_1, \cdots, x_n$ of $X$.
Let $E$ be a holomorphic vector bundle on $X$. Recall that a {\em quasi-parabolic structure\/} on $E$ over $S$ is a filtration of subspaces
\begin{equation}\label{filt}
  E_{x_i} = F_{i,1} \supsetneq \cdots \supsetneq F_{i,j} \supsetneq \cdots \supsetneq F_{i,\ell_i} \supsetneq F_{i,\ell_i+1} = 0
\end{equation}
over each point $x_i$ of $S$. A {\em parabolic structure\/} on $E$ over $S$ is a quasi-parabolic structure as above together with real numbers
\[
  0 \leq \alpha_{i,1} < \cdots < \alpha_{i,j} < \cdots < \alpha_{i,\ell_i} < 1\, ,
\]
which are called the {\em parabolic weights}. The weight $\alpha_{i,j}$ corresponds to the subspace $F_{i,j}$. (See \cite{Se77}, \cite[p.\ 67]{Se82}, \cite{MY92}.)
A {\em parabolic vector bundle\/} with parabolic divisor $S$ is a holomorphic vector bundle $E$ equipped with a quasi-parabolic structure over $S$ and parabolic weights as above.

For convenience, a parabolic vector bundle $(E, \{ F_{i,j} \}, \{ \alpha_{i,j} \})$ will be denoted by $E_\ast$.

We fix the divisor $S$ once and for all. Henceforth, the 
parabolic divisor for all parabolic vector bundles will be 
this $S$.

The {\em parabolic degree\/} of a parabolic vector bundle $E_\ast$ is defined to be
\[
  \pardeg(E_\ast) \,:= \, \degree(E) + \sum_{i=1}^n \sum_{j=1}^{\ell_i}
\alpha_{i,j}\cdot\dim(F_{i,j}/F_{i,j+1})
\]
and the real number
\[
  \parmu(E_\ast) \,:=\, \frac{\pardeg(E_\ast)}{\rank(E)}
\]
is called the {\em parabolic slope\/} of $E_\ast$.

See \cite{Bi97}, \cite{Yo95} for 
tensor product, dual and homomorphism bundles for parabolic
bundles.

\subsection{Orthogonal and symplectic structures}

Fix a parabolic line bundle $L_\ast$. The underlying holomorphic
line bundle will be denoted by $L$.

Let $E_\ast$ be a parabolic vector bundle, and let
\[
  \varphi\,:\, E_\ast \tensor E_\ast \,\longrightarrow \,L_\ast
\]
be a homomorphism of parabolic bundles. Tensoring both sides of this homomorphism with the parabolic dual $E_\ast^\ast$, we obtain a homomorphism
\[
  \varphi \tensor \id\,:\, E_\ast \tensor E_\ast \tensor 
E_\ast^\ast \,\longrightarrow\, L_\ast \tensor E_\ast^\ast\, .
\] 
Note that the sheaf of sections of the vector bundle underlying 
$E_\ast \tensor E_\ast^\ast$ is the sheaf of endomorphisms of $E$ 
preserving the quasi-parabolic filtrations. The trivial line 
bundle $\O_X$ equipped with the trivial parabolic structure 
(meaning there is no non-zero parabolic weight) is realized as a 
parabolic subbundle of $E_\ast \tensor E_\ast^\ast$ by sending any 
locally defined function $f$ to the locally defined endomorphism 
of $E$ given by pointwise multiplication with $f$. Let
\begin{equation} \label{isomorphism}
  \widetilde \varphi\,:\, E_\ast \,\longrightarrow \,
L_\ast \tensor E_\ast^\ast
\end{equation}
be the homomorphism defined by the composition
\[
  E_\ast \,=\, E_\ast \tensor \O_X \,\longhookrightarrow \,
E_\ast \tensor (E_\ast \tensor E_\ast^\ast)
\,=\, (E_\ast \tensor E_\ast) \tensor E_\ast^\ast
\,\stackrel{\varphi \tensor \id}{\longrightarrow} \,
L_\ast \tensor E_\ast^\ast\, .
\]

\begin{definition}\mbox{}
\begin{enumerate}
\item[(i)] An {\em 
orthogonal parabolic bundle\/} is a pair $(E_\ast\, , \varphi)$ of 
the above form such that $\varphi$ is symmetric, and the 
homomorphism $\widetilde \varphi$ in \eqref{isomorphism}
is an isomorphism.

\item[(ii)] A {\em symplectic parabolic bundle\/} is a pair 
$(E_\ast\, , \varphi)$ of the above form such that $\varphi$ is 
anti-symmetric, and the homomorphism $\widetilde \varphi$ is an 
isomorphism.
\end{enumerate}
\end{definition}

Let $(E_\ast\, , \varphi)$ be an orthogonal or symplectic
parabolic bundle with $E$ as the underlying vector bundle. Then
$E\otimes E$ is a coherent
subsheaf of the vector bundle underlying the parabolic tensor
product $E_\ast\otimes E_\ast$. Therefore, $\varphi$ produces
a homomorphism
\begin{equation}\label{issu2}
\widehat{\varphi}\, :\, E\otimes E\, \longrightarrow\, L\, ,
\end{equation}
where $L$ is the holomorphic line bundle underlying $L_\ast$.
A holomorphic subbundle
$$
F\, \subset\, E
$$
is called \textit{isotropic\/} if
\begin{equation}\label{issu}
\widehat{\varphi}(F\otimes F) \, =\, 0\, ,
\end{equation}
where $\widehat{\varphi}$ is constructed in \eqref{issu2}.

\subsection{Higgs fields}

Let $\Omega_X$ be the canonical line bundle of $X$. For notational
convenience, we write $\Omega_X(S) \,:=\,
\Omega_X \tensor \O_X(S)$. Let $E$ be a holomorphic vector bundle on $X$. A {\em logarithmic Higgs field\/} on $E$ is a holomorphic section
\[
  \theta \,\in\, H^0(X,\, \End(E) \tensor \Omega_X(S))\, .
\]

For every point $x_i \in S$, the fiber $(\Omega_X(S))_{x_i}$ is identified with $\CC$ using the Poincar\'e adjunction formula. The endomorphism
\[
  E_{x_i} \,\stackrel{\theta(x_i)}{\longrightarrow}\,
(E \tensor \Omega_X(S))_{x_i} \,=\, E_{x_i}
\]
is called the {\em residue\/} of $\theta$ at $x_i$; it will be denoted by $\Res(\theta, x_i)$.

\begin{definition}\label{d-oH}
\mbox{}
\begin{enumerate}
\item[(i)] Let $E_\ast$ be a parabolic vector bundle on $X$. A {\em parabolic Higgs field\/} on $E_\ast$ is a logarithmic Higgs field
\[
  \theta \,\in\, H^0(X, \,\End(E) \tensor \Omega_X(S))
\]
such that for every point $x_i \in S$, the residue $\Res(\theta, 
x_i)$ preserves the quasi-parabolic filtration in the sense that
\[
  \Res(\theta, x_i)(F_{i,j}) \,\subset \,
F_{i,j} \quad \text{for all } 1 \leq j \leq \ell_i
\]
(see \eqref{filt}).
\item[(ii)] A {\em parabolic Higgs bundle\/} is a pair $(E_\ast\, ,
\theta)$ consisting of a parabolic vector bundle $E_\ast$ and a parabolic Higgs field $\theta$ on $E_\ast$.
\end{enumerate}
\end{definition}

\begin{lemma} \label{higgs}
Let $E_\ast$ and $F_\ast$ be parabolic vector bundles
equipped with parabolic Higgs fields $\theta_E$ and
$\theta_F$, respectively. Then $\theta_E$ and $\theta_F$ together
induce a parabolic Higgs field on the parabolic tensor
product $E_\ast \tensor F_\ast$. Also, $\theta_E$ induces a
parabolic Higgs field on the parabolic dual $E_\ast^\ast$.
\end{lemma}

\begin{proof}
The logarithmic Higgs field $\theta_E$ on the vector bundle $E$ 
underlying $E_\ast$ induces a logarithmic Higgs field on the dual 
vector bundle $E^\ast$; this logarithmic Higgs field on $E^\ast$
will be denoted by $\theta_E'$. Let $E_0^\ast$ be the vector bundle 
underlying the parabolic dual $E_\ast^\ast$. Then $E_0^\ast$ is a 
subsheaf of $E^\ast$. It is straightforward to check that the 
logarithmic Higgs field $\theta_E'$ on $E^\ast$ produces a 
logarithmic Higgs field on $E_0^\ast$. This
logarithmic Higgs field on $E_0^\ast$ is a parabolic Higgs 
field on the parabolic vector bundle $E_\ast^\ast$.

Let $F$ be the vector bundle underlying $F_\ast$.
The two logarithmic Higgs fields $\theta_E$ and $\theta_F$
on $E$ and $F$ respectively
together induce a logarithmic Higgs field on the vector bundle $E 
\tensor F \tensor \O_X(S)$ (the Higgs field on $\O_X(S)$ is taken 
to be the zero section).
The vector bundle $(E_\ast \tensor F_\ast)_0$ 
underlying the parabolic tensor product $E_\ast \tensor F_\ast$ is 
a subsheaf of $E \tensor F \tensor \O_X(S)$. It is straightforward 
to check that the above logarithmic Higgs field on $E \tensor F 
\tensor \O_X(S)$ produces a logarithmic Higgs field on $(E_\ast 
\tensor F_\ast)_0$. This logarithmic Higgs field on $(E_\ast
\tensor F_\ast)_0$ is a parabolic Higgs field on the 
parabolic vector bundle $E_\ast \tensor F_\ast$.
\end{proof}

\begin{definition}\label{compatible-1}
Let $(E_\ast\, , \varphi)$ be an orthogonal or symplectic parabolic
bundle. A parabolic Higgs field $\theta$ on $E_\ast$ is said
to be {\em compatible with $\varphi$\/} if the isomorphism
$\widetilde \varphi$ in \eqref{isomorphism} takes
$\theta$ to the parabolic Higgs field on
$L_\ast\tensor E_\ast^\ast$ induced by $\theta$ (the
Higgs field on $L_\ast$ is taken to be the zero section).
\end{definition}

We will explain the above definition of a compatible parabolic 
Higgs field. Consider the pairing $\widehat\varphi$ in 
\eqref{issu2}. Since a Higgs field $\theta$ on $E_\ast$ is a 
section of $\End(E) \tensor \Omega_X(S)$, for any
holomorphic sections $s$ and $t$ of $E$ defined over
an open subset $U\, \subset\, X$, we have
$$
{\widehat\varphi}_\theta(s\, ,t)\, :=\,
\widehat\varphi(\theta(s)\otimes t) +
\widehat\varphi(s\otimes \theta(t))\, \in \,\Gamma(U,\,
L \tensor \Omega_X(S))\, .
$$
The Higgs field $\theta$ is compatible with $\varphi$
if and only if ${\widehat\varphi}_\theta(s\, ,t)\, =\, 0$
for all $s$ and $t$.

\begin{definition}\label{compatible}
An {\em orthogonal (respectively, symplectic) parabolic Higgs 
bundle\/} $(E_\ast\, , \varphi\, , \theta)$ is an orthogonal (respectively, 
symplectic) parabolic bundle $(E_\ast\, , \varphi)$ together with a 
parabolic Higgs field $\theta$ on $E_\ast$ which is compatible 
with $\varphi$.
\end{definition}

\section{Polystability}

Given a parabolic vector bundle $E_\ast$ on $X$ and a holomorphic subbundle $F$ of the underlying vector bundle $E$, we obtain an induced parabolic structure on $F$ by restricting the quasi-parabolic filtrations and the parabolic weights of $E$ to $F$. Let $F_\ast$ be the parabolic vector bundle obtained this way.

\begin{definition}
Let $(E_\ast\, , \theta)$ be a parabolic Higgs bundle on $X$.
\begin{enumerate}
\item[(i)] $(E_\ast\, , \theta)$ is called {\em stable\/} (respectively, {\em semistable}) if for every subbundle $F\, \subset\, E$ with $0 < \rank(F) < \rank(E)$ such that $\theta(F)\, \subset\, F\otimes \Omega_X(S)$ (see Definition \ref{d-oH}(i)), the inequality
\[
  \parmu(F_\ast) < \parmu(E_\ast) \quad \text{(respectively, } \parmu(F_\ast) \leq \parmu(E_\ast) \text{)}
\]
holds.
\item[(ii)] $(E_\ast\, , \theta)$ is called {\em polystable\/} if it is semistable and isomorphic to a direct sum of stable parabolic Higgs bundles.
\end{enumerate}
\end{definition}

Let $(E_\ast\, , \varphi\, , \theta)$ be
an orthogonal or symplectic parabolic Higgs bundle.
As before, the holomorphic vector bundle underlying $E_\ast$
will be denoted by $E$.

\begin{definition}\label{ds}
The orthogonal or symplectic parabolic Higgs bundle
$(E_\ast\, , \varphi\, , \theta)$ will be called {\em stable\/}
(respectively, {\em semistable}) if for every isotropic
subbundle $F\, \subset\, E$ of positive rank (see \eqref{issu})
such that $\theta(F)\, \subset\, F\otimes \Omega_X(S)$
(see Definition \ref{d-oH}(i)), the following condition holds:
\[
  \parmu(F_\ast) < \parmu(E_\ast) \quad \text{(respectively, }
\parmu(F_\ast) \leq \parmu(E_\ast) \text{)}\, .
\]
\end{definition}

Take any parabolic vector bundle $V_\ast$ on $X$. Using the natural
pairing of $V_\ast$ with its parabolic dual $V^\ast_\ast$, the parabolic
vector bundle $V_\ast\oplus (L_\ast \otimes V^\ast_\ast)$ is
equipped with
a symplectic as well as an orthogonal form with values in $L_\ast$.
To explain this, note that for any finite dimensional complex
vector space $W_0$, we have
$$
(W_0\oplus W^\ast_0)\otimes (W_0\oplus W^\ast_0)\, =\,
\bigwedge\nolimits^2(W_0\oplus W^\ast_0) \oplus \text{Sym}^2(
W_0\oplus W^\ast_0)\, ,
$$
and $\id_{W_0\oplus W^\ast_0}\, \in\,
\text{End}(W_0\oplus W^\ast_0) \, =\,
(W_0\oplus W^\ast_0)\otimes (W^\ast_0\oplus W_0)\, =
\, (W_0\oplus W^\ast_0)\otimes (W_0\oplus W^\ast_0)$
projects to a non-degenerate element in both
$\bigwedge^2(W_0\oplus W^\ast_0)$ and $\text{Sym}^2(
W_0\oplus W^\ast_0)$. Both the symplectic and orthogonal forms
on $V_\ast\oplus (L_\ast \otimes V^\ast_\ast)$ with values in
$L_\ast$ will be denoted by $\varphi_{V_\ast}$.

Let $\theta_V$ be a Higgs field on the parabolic vector bundle
$V_\ast$. The Higgs field on $V^\ast_\ast$ given by
Lemma \ref{higgs} will be denoted
by $\theta^\ast_V$.
The zero Higgs field on $L_\ast$ and $\theta^\ast_V$ together
define a Higgs field on $L_\ast \otimes V^\ast_\ast$
by Lemma \ref{higgs}; this
Higgs field on $L_\ast \otimes V^\ast_\ast$ will be denoted
by $\theta^L_V$. We note that the parabolic Higgs field
$\theta_V\oplus \theta^L_V$ on $V_\ast\oplus (L_\ast \otimes
V^\ast_\ast)$ is compatible with $\varphi_{V_\ast}$, so
\begin{equation}\label{dpH}
(V_\ast\oplus (L_\ast \otimes V^\ast_\ast)\, ,\varphi_{V_\ast}
\, , \theta_V\oplus \theta^L_V)
\end{equation}
is an orthogonal or symplectic parabolic Higgs bundle (depending
on whether $\varphi_{V_\ast}$ is the natural
orthogonal or symplectic form).

\begin{definition}
A semistable orthogonal (respectively, symplectic) parabolic
Higgs bundle
$(E_\ast\, , \varphi\, , \theta)$ will be called {\em polystable\/}
if it is a direct sum of finitely
many orthogonal (respectively, symplectic) parabolic
Higgs bundles
$$
(E_\ast\, , \varphi\, , \theta)\, =\,
\bigoplus_{i=1}^N \, (E^i_\ast\, , \varphi^i\, , \theta^i)\, ,
$$
where each $(E^i_\ast\, , \varphi^i\, , \theta^i)$ is either
stable (see Definition \ref{ds}) or it is of the form
$$(V_\ast\oplus (L_\ast \otimes V^\ast_\ast)\, ,\varphi_{V_\ast}
\, , \theta_V\oplus \theta^L_V)$$ (see \eqref{dpH}) with
$(V_\ast\, ,\theta_V)$ being a polystable parabolic Higgs bundle.
\end{definition}

To compare the above definition with the definition of
polystable orthogonal and symplectic parabolic vector bundles 
(without Higgs structure) 
given in \cite{BMW11}, note that a direct sum of polystable 
parabolic orthogonal (respectively, symplectic) Higgs bundles of 
the same parabolic slope is again polystable. Also, for two
parabolic vector bundles $V_\ast$ and $W_\ast$ with $V$ and
$W$ as the respective underlying vector bundles, we have
$$
(\theta_V\oplus \theta^L_V)\oplus (\theta_W\oplus \theta^L_W) 
\,=\, (\theta_{V\oplus W}\oplus \theta^L_{V\oplus W})\, ~
\text{ and }~\, \varphi_{V_\ast\oplus W_\ast}\,=\,
\varphi_{V_\ast}\, \oplus \varphi_{W_\ast}\, .
$$

\begin{proposition} \label{polystable}
Let $(E_\ast\, , \varphi\, , \theta)$ be a polystable orthogonal or
symplectic parabolic Higgs bundle. Then the parabolic Higgs
bundle $(E_\ast\, , \theta)$ is polystable.
\end{proposition}

\begin{proof}
Let $E$ be the vector bundle underlying $E_\ast$.
We will first show that $(E_\ast\, , \theta)$ is semistable.

Assume that $(E_\ast\, , \theta)$ is not semistable. Let
$$
F_\ast\, \subset\, E_\ast
$$
be the unique parabolic
subbundle of $E_\ast$ of positive rank such that
\begin{itemize}
\item $\theta(F)\, \subset\, F\otimes\Omega_X(S)$, where
$F\, \subset\, E$ is the subbundle underlying $F_\ast$,

\item $\parmu(F_\ast)\, \geq\, \parmu(V_\ast)$ for all
parabolic subbundles
$V_\ast\, \subset\, E_\ast$ with $\theta(V)\, \subset\,
V\otimes\Omega_X(S)$, where $V$ is the vector bundle underlying
$V_\ast$, and

\item $\text{rank}(F_\ast)$ is maximal among all parabolic
subbundles of 
$E_\ast$ satisfying the first two conditions.
\end{itemize}

The quotient bundle $E/F$ is equipped with a parabolic
structure given by the parabolic structure of $E_\ast$, and
this parabolic vector bundle is equipped
with a Higgs field given by $\theta$. If the parabolic Higgs
bundle $E/F$ equipped with these induced structures is not semistable,
we may consider the subbundle of it constructed as above using 
the three conditions. Proceeding inductively, we get a filtration 
of parabolic subbundles
\begin{equation}\label{hn}
0 \,=\, F^0_\ast\, \subset\, F_\ast \, =\, F^1_\ast
\, \subset\, F^2_\ast \, \subset\,\cdots \, \subset\,
F^{m-1}_\ast \, \subset\, F^m_\ast\,=\, E_\ast
\end{equation}
such that for all $i\, \in \, [1\, ,m]$,
\begin{itemize}
\item $\theta(F_i) \, \subset\, F_i\otimes\Omega_X(S)$,

\item $\parmu(F^i_\ast/F^{i-1}_\ast)\, \geq\,
\parmu(F'_\ast)$ for every parabolic subbundle
$F'_\ast\, \subset\, E_\ast/F^{i-1}_\ast$ preserved by the
Higgs field on $E_\ast/F^{i-1}_\ast$ induced by $\theta$, and

\item $F^i_\ast/F^{i-1}_\ast$ is of maximal rank among all
parabolic subbundles of $E_\ast/F^{i-1}_\ast$ satisfying
the first two conditions. 
\end{itemize}
The filtration in \eqref{hn} is called the \textit{Harder--Narasimhan\/}
filtration for $(E_\ast\, , \theta)$.

The Higgs field $\theta$ induces a Higgs field $\theta'$
on $L_\ast \tensor E_\ast^\ast$ using the zero Higgs field
on $L_\ast$ (see Lemma \ref{higgs}).
We note that $(L_\ast \tensor E_\ast^\ast
\, , \theta')$ is not semistable because
$(E_\ast\, , \theta)$, which is isomorphic to
it (see Definition \ref{compatible-1}), is not semistable. Let
\begin{equation}\label{hn2}
0 \,=\, G^0_\ast\, \subset\, G^1_\ast
\, \subset\, G^2_\ast \, \subset\,\cdots \, \subset\,
G^{m-1}_\ast \, \subset\, G^m_\ast\,=\, L_\ast \tensor E_\ast^\ast
\end{equation}
be the Harder--Narasimhan filtration for
$(L_\ast \tensor E_\ast^\ast \, , \theta')$. From the uniqueness of
the Harder--Narasimhan filtration we conclude that
\begin{equation}\label{uhf}
\widetilde{\varphi} (F^i_\ast) \, =\, G^i_\ast
\end{equation}
for all $i$, where $\widetilde{\varphi}$ is the isomorphism in
\eqref{isomorphism}.

We put down some properties of the parabolic slope
which are straightforward to derive.
\begin{itemize}
\item $\parmu(W_\ast \otimes W'_\ast) \, =\,
\parmu(W_\ast) + \parmu(W'_\ast)$ for any parabolic
vector bundles $W_\ast$ and $W'_\ast$; also,
$\parmu(W^\ast_\ast)\, =\, - \parmu(W_\ast)$.

\item A parabolic Higgs bundle $(W_\ast\, ,\beta)$ is
semistable (respectively, polystable)
if and only if $W^\ast_\ast$ equipped with the Higgs field
induced by $\beta$ is semistable (respectively, polystable).
(This follows from the first property.)

\item Let $(M_\ast\, ,\gamma)$ be a parabolic Higgs
line bundle. A parabolic Higgs vector bundle $(W_\ast\, ,\beta)$
is semistable (respectively, polystable) if and only if
$W_\ast\otimes M_\ast$ equipped
with the Higgs field induced by $\beta$ and $\gamma$
(see Lemma \ref{higgs})
is semistable (respectively, polystable). (This also
follows from the first property.)
\end{itemize}

{}From the above properties it follows immediately
that the filtration in \eqref{hn2} is given by the dual
of the filtration in \eqref{hn}. This means that
\begin{equation}\label{is1}
(L_\ast \tensor E_\ast^\ast)/G^{m-i}_\ast \,=\, 
L_\ast \tensor (F^i_\ast)^\ast 
\end{equation}
for all $i\, \in\, [1\, ,m]$.

Combining \eqref{is1} and \eqref{uhf} it follows
that $F^1_\ast$ is an isotropic subbundle
of $E_\ast$ for the pairing $\varphi$, because the
composition
$$
F^1_\ast \, \stackrel{\widetilde{\varphi}}{\longrightarrow}\,
G^1_\ast\, \longrightarrow\, (L_\ast \tensor E_\ast^\ast)/
G^{m-1}_\ast
$$
vanishes identically (recall that $m\, \geq\, 2$).
Since $F_\ast\,=\, F^1_\ast$ (see \eqref{hn}) is an isotropic 
subbundle for the pairing $\varphi$,
the subbundle $F_\ast$ violates the semistability
condition for $(E_\ast\, , \varphi\, , \theta)$. But we
know that $(E_\ast\, , \varphi\, , \theta)$ is semistable
because it is polystable. In view of this contradiction,
we conclude that $(E_\ast\, , \theta)$ is semistable.

Assume that $(E_\ast\, , \theta)$ is not polystable.

Consider all parabolic subbundles
$$
V_\ast\, \subset\, E_\ast
$$
such that
\begin{itemize}
\item $\parmu(V_\ast)\, =\, \parmu(E_\ast)$, 

\item $\theta(V)\, \subset\, V\otimes\Omega_X(S)$, where
$V\, \subset\, E$ is the subbundle underlying $V_\ast$, and

\item the parabolic Higgs bundle defined by $V_\ast$ equipped
with the Higgs field induced by $\theta$ is polystable.
\end{itemize}

Let $\widehat{V}_\ast$ be the parabolic subbundle of
$E_\ast$ generated by all such parabolic subbundles. From
the construction of $\widehat{V}_\ast$ it follows immediately
that
$$
\theta(\widehat{V})\, \subset\,
\widehat{V}\otimes\Omega_X(S)\, ,
$$
where $\widehat{V}\, \subset\, E$ is the subbundle underlying
$\widehat{V}_\ast$. Let $\theta_{\widehat{V}}$ be the
Higgs field on $\widehat{V}_\ast$ defined by~$\theta$.
We have
\begin{itemize}
\item $\parmu(\widehat{V}_\ast)\, =\, \parmu(E_\ast)$,

\item the parabolic Higgs bundle $(\widehat{V}_\ast\, ,
\theta_{\widehat{V}})$ is polystable, and

\item $\widehat{V}_\ast$ is of maximal rank among all
parabolic subbundles of $E_\ast$ satisfying the first
two conditions.
\end{itemize}
These follow from the fact that for any two parabolic
subbundles $W^1_\ast$ and $W^2_\ast$ of $E_\ast$ preserved
by $\theta$ and satisfying the two conditions
\begin{itemize}
\item $\parmu(W^j_\ast)\, =\, \parmu(E_\ast)$, $j\, =\,1\, ,2$,
and

\item $W^j_\ast$ equipped with the parabolic Higgs field induced
by $\theta$ is polystable,
\end{itemize}
the parabolic subbundle of $E_\ast$ generated by $W^1_\ast$ and 
$W^2_\ast$ is preserved by $\theta$ and also satisfies
the above two conditions. (See \cite[p.\ 23, Lemma 1.5.5]{HL}.)

Since $(E_\ast\, , \theta)$ is assumed to be not polystable,
we have $\text{rank}(\widehat{V}_\ast)\, < \,
\text{rank}(E_\ast)$.

This polystable parabolic Higgs bundle $(\widehat{V}_\ast\, ,
\theta_{\widehat{V}})$ is called the \textit{socle\/} of
$(E_\ast\, , \theta)$. Just as in the case of the Harder--Narasimhan
filtration, we get a filtration of parabolic
subbundles
\begin{equation}\label{soc}
0 \,=\, V^0_\ast\, \subset\, V_\ast \, =\, V^1_\ast
\, \subset\, V^2_\ast \, \subset\,\cdots \, \subset\,
V^{n-1}_\ast \, \subset\, V^n_\ast\,=\, E_\ast
\end{equation}
such that for all $i\, \in \, [1\, ,n]$,
\begin{itemize}
\item $\theta(V_i) \, \subset\, V_i\otimes\Omega_X(S)$, and

\item $V^i_\ast/V^{i-1}_\ast$ equipped with the
Higgs structure induced by $\theta$
is the socle of $E_\ast/V^{i-1}_\ast$ equipped with the
Higgs structure induced by $\theta$.
\end{itemize}
The filtration in \eqref{soc} is called the \textit{socle
filtration\/} for $(E_\ast\, , \theta)$.

We note that $(L_\ast \tensor E_\ast^\ast
\, , \theta')$ is semistable because it is isomorphic
to the semistable parabolic Higgs bundle
$(E_\ast\, , \theta)$. Let
\begin{equation}\label{soc2}
0 \,=\, W^0_\ast\, \subset\, W^1_\ast
\, \subset\, W^2_\ast \, \subset\,\cdots \, \subset\,
W^{n-1}_\ast \, \subset\, W^n_\ast\,=\, L_\ast \tensor E_\ast^\ast
\end{equation}
be the socle filtration for
$(L_\ast \tensor E_\ast^\ast \, , \theta')$. From the uniqueness of
the socle filtration it follows that
\begin{equation}\label{isb}
\widetilde{\varphi} (V^i_\ast) \, =\, W^i_\ast
\end{equation}
for all $i$, where $\widetilde{\varphi}$ is the isomorphism in
\eqref{isomorphism}.

{}From the properties of the parabolic slope listed above it follows
that
$$
(L_\ast \tensor E_\ast^\ast)/W^{n-i}_\ast \,=\,
L_\ast \tensor (V^i_\ast)^\ast
$$
for all $i\, \in\, [1\, ,n]$. Just as before, this and
\eqref{isb} together imply that
$V^1_\ast\, \subset\, E_\ast$ is an isotropic subbundle for
$\varphi$.

Consider $V^1_\ast\, =\, V_\ast$ in \eqref{soc}.
Let $\theta_{V}$ be the Higgs field on $V_\ast$
induced by $\theta$.
Since $V_\ast\, \subset\, E_\ast$ is an isotropic subbundle for
$\varphi$ with $\parmu(V_\ast)
\, =\, \parmu(E_\ast)$ and preserved by $\theta$, from the
definition of polystability
of $(E_\ast\, , \varphi\, , \theta)$ it follows that
there is an orthogonal or symplectic parabolic
Higgs bundle $(W_\ast\, ,\phi\, ,\alpha)$ (depending on whether
$(E_\ast\, , \varphi)$ is orthogonal or symplectic) such that
\begin{equation}\label{dec}
(E_\ast\, ,\varphi\, ,\theta)\, =\,
(V_\ast\oplus (L_\ast \otimes V^\ast_\ast)\, ,\varphi_{V_\ast}
\, , \theta_V\oplus \theta^L_V)\oplus
(W_\ast\, ,\phi\, ,\alpha)\, ,
\end{equation}
where $(V_\ast\oplus (L_\ast \otimes V^\ast_\ast)\, ,\varphi_{V_\ast}
\, , \theta_V\oplus \theta^L_V)$ is defined in \eqref{dpH}.

We have shown that the parabolic Higgs bundle
$(E_\ast\, ,\theta)$ is semistable.
Therefore, from \eqref{dec} it follows that both the parabolic
Higgs bundles $(V_\ast\oplus (L_\ast \otimes V^\ast_\ast)
\, , \theta_V\oplus \theta^L_V)$ and $(W_\ast\, ,\alpha)$
are semistable.

We note that the parabolic Higgs bundle
$(L_\ast \otimes V^\ast_\ast \, ,  \theta^L_V)$ is polystable,
because $(V_\ast\, ,\theta_V)$ is polystable. Also, we
have $\parmu(L_\ast \otimes V^\ast_\ast) \,=\,
\parmu(E_\ast)$, because $\parmu(V_\ast) \,=\, \parmu(E_\ast)$.
Since $(V_\ast\, ,\theta_V)$ is the socle of
$(E_\ast\, ,\theta)$, these imply that the Higgs parabolic
subbundle
$$
(L_\ast \otimes V^\ast_\ast \, ,  \theta^L_V)\, \subset\,
(E_\ast\, ,\theta)
$$
is actually contained in $(V_\ast\, ,\theta_V)$. But this
contradicts \eqref{dec}. Therefore, we conclude that
$(E_\ast\, , \theta)$ is polystable.
\end{proof}

\section{Hermitian-Einstein connections}

Fix a Hermitian metric $\omega$ on $X \setminus S$ which extends 
smoothly over $X$; it is K\"ahler because $\dim_\CC X = 1$.

\begin{definition}
A {\em Hermitian-Einstein metric\/} on a Higgs vector bundle
$(E\, , \theta)$ over $X \setminus S$ is defined to be a Hermitian metric $h$ on $E$ such that its Chern curvature form $F_h$ satisfies the equation
\begin{equation} \label{hermitian-einstein}
\sqrt{-1}\cdot \Lambda_\omega (F_h + [\theta, \theta^\ast])\,=
\,\lambda \cdot \id_E
\end{equation}
for some $\lambda \,\in\, \RR$ which is known as
the {\em Einstein factor}; here, $\Lambda_\omega$ is
the adjoint of forming the wedge product with $\omega$ and
$\theta^\ast$ is the adjoint endomorphism of $\theta$ with respect to
$h$, and $[{\cdot}\, , {\cdot}]$ is defined using the exterior
product on forms and the Lie algebra structure of the fibers
of $\End(E)$.

If $h$ is a Hermitian-Einstein metric, then its Chern connection
is called a {\em Hermitian-Einstein connection}.
\end{definition}

Let $(E_\ast\, , \theta)$ be a parabolic Higgs bundle. In
\cite[Theorem 4]{Si90}, Simpson describes a construction of a
background metric on $E$ over $X \setminus S$ from the given data
$(E_\ast\, , \theta)$, which is compatible with taking parabolic duals.
Also, it is compatible with taking parabolic tensor products up
to mutual boundedness of the resulting background metrics. (See
\cite[Proposition 3.1, Corollary 4.3, Theorem 4]{Si90}.) The metric
on $E$ over $X \setminus S$ obtained from $(E_\ast\, , \theta)$ via
this construction will be denoted by $h_0(E_\ast\, , \theta)$.

The following existence result is known (see \cite[Lemma 6.3, Theorem 6]{Si90}, \cite[Theorem 1]{Si88}):

\begin{theorem} \label{simpson}
If $(E_\ast\, , \theta)$ is stable, then there is a Hermitian-Einstein metric $h$ on $(E\, , \theta)$ over $X \setminus S$ such that the metric $h$ and the background metric $h_0(E_\ast\, , \theta)$ are mutually bounded.
\end{theorem}

We will prove the uniqueness of the associated Hermitian-Einstein
connection in Theorem \ref{simpson}. For this, we need the
following lemma:

\begin{lemma} \label{parallel} Let $(F\, , \widetilde \theta)$ be 
a Higgs vector bundle on $X \setminus S$ admitting a 
Hermitian-Einstein metric $h$ with Einstein factor $\lambda = 0$.
Let $\sigma$ be a holomorphic section of $F$ satisfying
the conditions that $\widetilde 
\theta(\sigma) \,=\, 0$ and $\sigma$ is bounded with respect to $h$. 
Then $\sigma$ is parallel with respect to the Chern connection $D$ of $h$.
\end{lemma}

\begin{proof}
Let $\square \,:=\, \sqrt{-1} \cdot \Lambda_\omega \dbar \partial$ be
the (complex) Laplacian on functions with respect to~$\omega$. By
\cite[Proposition 2.4]{Si88}, we know that the manifold $X \setminus S$ satisfies the following condition:
\begin{equation} \label{laplace-condition}
\begin{minipage}{14cm} \centering
If $f$ is a bounded non-negative smooth function on
$X \setminus S$, \\ then $\square f \,\leq\, 0$ implies
$\square f \,= \,0$.
\end{minipage}
\end{equation}
We want to apply \eqref{laplace-condition} to the function $f \,:=\,
|\sigma|_h^2$. Since $\sigma$ is holomorphic, we have
\[
  \square |\sigma|_h^2 \,=\,
h\big((\sqrt{-1} \cdot \Lambda_\omega F_h)(\sigma), \sigma\big) -
|D' \sigma|_h^2\, ,
\]
where $D'$ is the $(1,0)$ component of the Chern
connection $D$. As the Einstein factor of $h$ is $0$,
the Hermitian-Einstein equation \eqref{hermitian-einstein} implies that
\[
  \sqrt{-1}\cdot\Lambda_\omega F_h \,= \,
- \sqrt{-1}\cdot\Lambda_\omega [\widetilde{\theta}\, ,
\widetilde{\theta}^\ast]\, .
\]
Combining this with $\widetilde \theta(\sigma) \,=\, 0$ it
follows that
\[
  h\big((\sqrt{-1} \cdot \Lambda_\omega F_h)(\sigma), \sigma\big)\,
 =\, - \sqrt{-1} \cdot \Lambda_\omega h\big([\widetilde \theta\, ,
\widetilde \theta^\ast](\sigma), \sigma\big)
\,=\, - |\widetilde \theta^\ast(\sigma)|_h^2 \leq 0\, .
\]
Consequently, we have
\[
  \square |\sigma|_h^2\,\leq\, - |D' \sigma|_h^2 \,\leq\, 0\, .
\]
Since $|\sigma|_h^2$ is bounded, condition \eqref{laplace-condition}
yields $\square |\sigma|_h^2 \,= \,0$, and thus $D' \sigma \,=\, 0$.
As $\sigma$ is holomorphic, this already implies that
$D \sigma \,=\, 0$.
\end{proof}

\begin{proposition} \label{uniqueness}
Let $h_1$ and $h_2$ be two Hermitian-Einstein metrics on $(E\, ,
\theta)$ over $X \setminus\nolinebreak S$ which are mutually bounded. Then the corresponding Chern connections agree.
In particular, the Hermitian-Einstein connection on $(E\, , \theta)$ over $X \setminus S$ given by Theorem~\ref{simpson} is unique.
\end{proposition}

\begin{proof}
Let $F \,:=\, \End(E)\,= \, E \tensor E^\ast$ be the endomorphism
bundle of $E$ over $X \setminus S$ equipped with the Higgs
field $\widetilde \theta$ induced by $\theta$. Let $h$ be the
Hermitian metric on $F$ induced by $h_1$ and $h_2$. Then $h$
is a Hermitian-Einstein metric on $F$ with Einstein factor $\lambda = 0$. Its
Chern connection is
\[
  D \,=\, D_1 \tensor \id_{E^\ast} + \id_E \tensor D_2^\ast\, ,
\]
where $D_1$ and $D_2$ are the Chern connections associated to 
$h_1$ and $h_2$, respectively, and $D_2^\ast$ is the connection on 
$E^\ast$ induced by $D_2$.

We want to apply Lemma \ref{parallel} 
to the holomorphic section $\sigma \,:=\, \id_E$ of $\End(E)$. Since 
$\id_E$ commutes with $\theta$, we have $\widetilde \theta(\id_E) 
\,=\, 0$. The mutual boundedness of $h_1$ and $h_2$ implies that 
$\id_E$ is bounded with respect to $h$. Lemma \ref{parallel} then 
yields
\[
  0 \,=\, D(\id_E) \,= \, D_1 \circ \id_E - \id_E \circ D_2\, ,
\]
and thus $D_1 = D_2$.
\end{proof}

Now let $(E_\ast\, , \varphi\, , \theta)$ be an orthogonal or 
symplectic parabolic Higgs bundle on $X$. Let $L_\ast$ be the 
parabolic line bundle that we fixed earlier. Since $L_\ast$ is 
stable, it admits a Hermitian-Einstein metric $h_L$ by Theorem 
\ref{simpson} (the Higgs field on $L_\ast$ is always taken to be the
zero section). The corresponding Hermitian-Einstein
connection on $L_\ast$ is unique 
by Proposition \ref{uniqueness}; denote this connection by 
$\nabla_L$.

\begin{definition}
A {\em Hermitian-Einstein connection\/} on $(E_\ast\, , \varphi\ ,
\theta)$ is a Hermitian-Einstein connection $D$ on the underlying
Higgs vector bundle $(E\, , \theta)$ over $X \setminus S$ such
that the isomorphism $\widetilde \varphi$ in \eqref{isomorphism}
takes $D$ to the connection on $L_\ast \tensor E_\ast^\ast$ induced
by $\nabla_L$ and the dual connection $D^\ast$ on
$E_\ast^\ast$ for $D$.
\end{definition}

\begin{theorem} \label{result}
Let $(E_\ast\, , \varphi\, , \theta)$ be a polystable orthogonal or
symplectic parabolic Higgs bundle on $X$. Then $(E_\ast\, ,
\varphi\, , \theta)$ admits a Hermitian-Einstein connection.
\end{theorem}

\begin{proof}
Since $(E_\ast\, , \varphi\, , \theta)$ is polystable, the
parabolic Higgs bundle $(E_\ast\, , \theta)$ is polystable by Proposition \ref{polystable}. By Theorem \ref{simpson}, there is a Hermitian-Einstein metric $h$ on $(E\, , \theta)$ over $X \setminus S$ such that $h$ and $h_0(E_\ast\, , \theta)$ are mutually bounded.

We have to show that the Chern connection $D$ on $E_\ast$
associated to $h$ is 
a Hermitian-Einstein connection on $(E_\ast\, , \varphi\, , \theta)$.
By Lemma \ref{higgs}, the Higgs field
$\theta$ induces a Higgs field $\widetilde 
\theta$ on the parabolic vector bundle $L_\ast \tensor 
E_\ast^\ast$. The Hermitian metric $h'$ on $L_\ast \tensor 
E_\ast^\ast$ induced by $h_L$ and $h$ is a Hermitian-Einstein 
metric on $(L_\ast \tensor E_\ast^\ast\, , \widetilde \theta)$. As 
$h$ and $h_0(E_\ast\, , \theta)$ are mutually bounded, and the 
construction of the background metric is compatible with taking 
duals and tensor products up to mutual boundedness, it follows 
that $h'$ and $h_0(L_\ast \tensor E_\ast^\ast\, , \widetilde \theta)$ 
are mutually bounded.

On the other hand, the Hermitian metric $h''$ on $L_\ast \tensor E_\ast^\ast$
given by the isomorphism $\widetilde \varphi$ in 
\eqref{isomorphism} is also a Hermitian-Einstein metric on 
$(L_\ast \tensor E_\ast^\ast\, , \widetilde \theta)$ because the Higgs 
field $\theta$ is compatible with the orthogonal or symplectic 
structure $\varphi$ (see Definition \ref{compatible}). As 
$\widetilde \varphi$ is an isomorphism of parabolic bundles, the 
metrics $h''$ and $h_0(L_\ast \tensor E_\ast^\ast\, , \widetilde 
\theta)$ are mutually bounded.

By Proposition \ref{uniqueness} it follows that the
Chern connections of $h'$ and $h''$ coincide. This means that the
Chern connection $D$ associated to $h$ is a Hermitian-Einstein connection on $(E_\ast\, , \varphi\, , \theta)$.
\end{proof}

There is also a converse to Theorem \ref{result}. For this, one has
to impose a condition on the asymptotic behavior of the Hermitian-Einstein connection near the parabolic divisor $S$. In \cite{Po93}, Poritz defines a space
\begin{equation} \label{poritz}
  \A \,= \, \A_\D^\delta
\end{equation}
of connections depending on the parabolic structure of $E_\ast$ (see \cite[Definition 3.2]{Po93}). Using this definition, we have:

\begin{proposition} \label{converse}
Let $(E_\ast\, , \varphi\, , \theta)$ be an orthogonal or symplectic
parabolic Higgs bundle on $X$. If $E$ admits a
Hermitian-Einstein connection lying in the space $\A$, then
it is polystable.
\end{proposition}

\begin{proof}
The proof of \cite[Theorem 6.4]{Po93} immediately generalizes to our situation.
\end{proof}

\end{document}